\documentclass[10pt]{article}
\usepackage{amsthm}
\usepackage{amsmath}
\usepackage{amssymb}
\usepackage{makeidx}

\theoremstyle{definition}
\newtheorem{definition}{Definition}[section]

\newtheorem{remark}[definition]{Remark}

\theoremstyle{plain}
\newtheorem{theorem}[definition]{Theorem}
\newtheorem{lemma}[definition]{Lemma}
\newtheorem{proposition}[definition]{Proposition}
\newtheorem{corollary}[definition]{Corollary}

\numberwithin{equation}{section}

\newcommand{\BI}{\mathbb{BI}}
\newcommand{\KL}{\mathbb{KL}}

\newcommand{\BZ}{\mathbb{BZL}}
\newcommand{\AOL}{\mathbb{AOL}}

\newcommand{\PBZ}{PBZ$^{\ast }$}

\def\N{{\mathbb N}}

\def\I{{\mathbb I}}

\def\V{{\mathbb V}}
\def\W{{\mathbb W}}

\def\S{\mathcal{S}}

\def\B{\mathrm{B}}

\begin{document}
\title{A Note on Congruences of Infinite Bounded Involution Lattices}
\author{Claudia MURE\c SAN\thanks{Dedicated to the memory of my beloved grandmother, Elena Mircea}\\ 
{\small University of Cagliari, University of Bucharest}\\ 
{\small c.muresan@yahoo.com, cmuresan@fmi.unibuc.ro}}
\date{\today }
\maketitle

\begin{abstract} We prove that an infinite (bounded) involution lattice and even pseudo--Kleene algebra can have any number of congruences between $2$ and its number of elements or equalling its number of subsets, regardless of whether it has as many ideals as elements or as many ideals as subsets; consequently, the same holds for antiortholattices. Under the Generalized Continuum Hypothesis, this means that an infinite (bounded) involution lattice, pseudo--Kleene algebra or antiortholattice can have any number of congruences between $2$ and its number of subsets, regardless of its number of ideals.


{\em $2010$ Mathematics Subject Classification:} primary: 06B10; secondary: 06F99, 06D30.\end{abstract}

\section{Introduction}
\label{introduction}

As part of our main result from \cite{gccm}, we have proven that, under the Generalized Continuum Hypothesis, an infinite lattice can have any number of congruences between $2$ and its number of subsets. In this paper, we prove that the same holds for infinite (bounded) involution lattices and even infinite pseudo--Kleene algebras, thus also for infinite antiortholattices, which are algebraic structures with pseudo--Kleene algebra reducts originating in the study of quantum logics \cite{GLP,PBZ2,rgcmfp,pbzsums,eucardbi}; moreover, we can let these algebras with lattice reducts have any numbers of ideals. Furthermore, if we restrict to cardinal numbers that are either smaller than the numbers of elements or equal to the numbers of subsets of these algebras, then we do not need to enforce the Continuum Hypothesis.

Note that, in the finite case, this result on numbers of congruences does not hold, due to the limited number of configurations. The finite case for lattices  has been treated in \cite{gcze,kumu}, the one for semilattices in \cite{gcz}, and the one for involution lattices, pseudo--Kleene algebras and antiortholattices in \cite{eucardbi}.

\section{Preliminaries}
\label{preliminaries}

We denote by $\N $ the set of the natural numbers and by $\N ^*=\N \setminus \{0\}$. $\amalg $ denotes the disjoint union. For any set $M$, we denote by $|M|$ the cardinality of $M$, by ${\cal P}(M)$ the set of the subsets of $M$ and, if $M$ is nonempty, by ${\rm Part}(M)$ and $({\rm Eq}(M),\vee ,\cap ,\Delta _M,\nabla _M)$ the bounded lattices of the partitions and the equivalences of $M$, respectively, and by $eq:{\rm Part}(M)\rightarrow {\rm Eq}(M)$ the canonical lattice isomorphism; for any finite partition $\{M_1,\ldots ,M_n\}$, $eq(\{M_1,\ldots ,M_n\})$ will be streamlined to $eq(M_1,\ldots ,M_n)$.

All algebras will be designated by their underlying sets. By {\em trivial algebra} we mean one--element algebra, and by {\em simple algebra} we mean algebra with at most two congruences. For any $n\in \N ^*$, ${\cal L}_n$ will denote the $n$--element chain. Let $L$ be a (bounded) lattice. Then the dual of $L$ will be denoted by $L^d$. The sets of the filters, principal filters, prime filters, ideals, principal ideals and prime ideals of $L$ will be denoted by ${\rm Filt}(L)$, ${\rm PFilt}(L)$, ${\rm Spec}_{\rm Filt}(L)$, ${\rm Id}(L)$, ${\rm PId}(L)$ and ${\rm Spec}_{\rm Id}(L)$, respectively. Recall that the prime ideals of $L$ are exactly the set complements of its prime filters and that, if $L$ is a chain, then all its proper filters are prime and the same goes for its ideals, hence the proper ideals of $L$ are exactly the set complements of its proper filters, in particular $|{\rm Filt}(L)|=|{\rm Id}(L)|$. For any $X\subseteq L$ and any $a,b\in L$, we denote by $[X)_L$ and $[a)_L$ the filter of $L$ generated by $X$ and by $a$, respectively, by $(X]_L$ and $(a]_L$ the ideal of $L$ generated by $X$ and by $a$, respectively, and by $[a,b]_L=[a)_L\cap (b]_L$.

Let $\V $ be a variety of algebras of a similarity type $\tau $ and $A$ and $B$ be algebras with reducts belonging to $\V $. Following \cite{rgcmfp,eucardbi}, we denote by $A\cong _{\V }B$ the fact that the $\tau $--reducts of $A$ and $B$ are isomorphic, and by ${\rm Con}_{\V }(A)$ and $\S _{\V }(A)$ the sets of the congruences and the subalgebras of the $\tau $--reduct of $A$, respectively. Recall that ${\rm Con}_{\V }(A)$ is a complete sublattice of ${\rm Eq}(A)$ {\rm \cite[Corollary 2, p. 51]{gralgu}, from which it follows that, if $\sigma $ is a similarity type of reducts of $\tau $--algebras and $\W $ is a variety of algebras of type $\sigma $, then ${\rm Con}_{\V }(A)$ is a complete bounded sublattice of ${\rm Con}_{\W }(A)$. Obviously, for any $\theta \in {\rm Con}_{\V }(A)$ and any $S\in \S _{\V }(A)$, we have $\theta \cap S^2\in {\rm Con}_{\V }(S)$. If $n\in \N ^*$ and $\tau $ contains constants $\kappa _1,\ldots ,\kappa _n$, then we denote by ${\rm Con}_{\V \kappa _1\ldots \kappa _n}(A)=\{\theta \in {\rm Con}_{\V }(A)\ |\ \kappa _1^A/\theta =\{\kappa _1^A\},\ldots ,\kappa _n^A/\theta =\{\kappa _n^A\}\}$, which is easily seen to be a complete sublattice of ${\rm Con}_{\V }(A)$ and thus a bounded lattice. If $\V $ is the variety of lattices or that of bounded lattices, then we eliminate the index $_{\V }$ from the previous notations.

\section{Lattices with Involutions and Some Constructions of Lattices}
\label{thealg}

\begin{definition} We call a {\em lattice with involution} or {\em involution lattice} (in brief, {\em i--lattice}) an algebra $(L,\vee ,\wedge ,\cdot ^{\prime })$ of type $(2,2,1)$, where $(L,\vee ,\wedge )$ is a lattice and $\cdot ^{\prime }$ is an order--reversing operation such that $a^{\prime \prime }=a$ for all $a\in L$, called {\em involution}.

A {\em bounded involution lattice} (in brief, {\em bi--lattice}) is an algebra $(L,\vee ,\wedge ,\cdot ^{\prime },0,1)$ of type $(2,2,1,0,0)$, where $(L,\vee ,\wedge ,0,1)$ is a bounded lattice and $(L,\vee ,\wedge ,\cdot ^{\prime })$ is an i--lattice. 

Distributive bi--lattices are called {\em De Morgan algebras}.

We consider the following condition on a bi--lattice $L$:

\begin{tabular}{ll}
\textcircled{k} & for all $a,b\in L$, $a\wedge a^{\prime }\leq b\vee b^{\prime }$\end{tabular}

A {\em pseudo--Kleene algebra} is a bi--lattice that satisfies condition \textcircled{k}. The involution of a pseudo--Kleene algebra is called {\em Kleene complement}.

Distributive pseudo--Kleene algebras are called {\em Kleene algebras} or {\em Kleene lattices}.

A bi--lattice $L$ is said to be {\em paraorthomodular} iff, for all $a,b\in L$, if $a\leq b$ and $a^{\prime }\wedge b=0$, then $a=b$.\label{bilat}\end{definition}

We will denote by $\I $, $\BI $ and $\KL $ the variety of involution lattices, bounded involution lattices and pseudo--Kleene algebras, respectively. An i--lattice with underlying set $L$ and involution $\cdot ^{\prime }$ will often be designated by $(L,\cdot ^{\prime })$. Unless specified otherwise, the involution of an i--lattice will be denoted $\cdot ^{\prime }$. Obviously, the involution of any i--lattice $L$ is a dual lattice automorphism of $L$, hence $L$ is self--dual and thus it has $|{\rm Filt}(L)|=|{\rm Id}(L)|$.

Of course, any Boolean algebra $A$ is a Kleene lattice, with the involution equalling its Boolean complement, which is preserved by all its lattice congruences, so that ${\rm Con}_{\I }(A)={\rm Con}(A)$. Remember that Boolean algebras are exactly the distributive orthomodular lattices. Furthermore, any orthomodular lattice $L$ is a paraorthomodular pseudo--Kleene algebra with all its lattice congruences preserving its involution, so that ${\rm Con}_{\I }(L)={\rm Con}(L)$ \cite{bruhar}.

\begin{definition}{\rm \cite{GLP,PBZ2,rgcmfp,pbzsums}} A {\em Brouwer--Zadeh lattice} (in brief, {\em BZ--lattice}) is an algebra $(L,\vee ,\wedge ,\cdot ^{\prime },\cdot ^{\sim },0,1)$ of type $(2,2,1,1,0,0)$ such that $(L,\vee ,\wedge ,\cdot ^{\prime },0,1)$ is a pseudo--Kleene algebra and the unary operation $\cdot ^{\sim }$, called {\em Brouwer complement}, is order--reversing and satisfies $a\wedge a^{\sim }=0$ and $a\leq a^{\sim \sim }=a^{\sim \prime }$ for all $a\in L$.

The Brouwer complement on a BZ--lattice $L$ defined by $0^{\sim }=1$ and $a^{\sim }=0$ for all $a\in L\setminus \{0\}$ is called the {\em trivial Brouwer complement}.

A {\em \PBZ --lattice} is a paraorthomodular BZ--lattice $L$ that satisfies the following condition, for all $a\in L$: $(a\wedge a^{\prime })^{\sim }=a^{\sim }\vee a^{\prime \sim }$.

An {\em antiortholattice} is a \PBZ --lattice with the property that $0$ and $1$ are its only elements whose Kleene complements are bounded lattice complements.\end{definition}

We denote by $\BZ $ the variety of BZ--lattices. \PBZ --lattices form a variety, as well. However, antiortholattices form a proper universal class, denoted by $\AOL $. Antiortholattices are exactly the \PBZ --lattices whose Brouwer complement is trivial. See \cite{GLP,PBZ2,rgcmfp,pbzsums} for these properties.

We now recall the definition of the horizontal sum of a  family of nontrivial bounded lattices, obtained by glueing those lattices at their bottom elements and at their top elements. Let $(L_i,\leq ^{L_i},0^{L_i},1^{L_i})_{i\in I}$ be a nonempty family of nontrivial bounded lattices. Then the {\em horizontal sum} of the family $(L_i,\leq ^{L_i},0^{L_i},1^{L_i})_{i\in I}$ is the bounded lattice $(\boxplus _{i\in I}L_i,\leq ,0,1)$ defined as follows: let $\displaystyle L=\amalg _{i\in I}L_i$ and $\varepsilon $ the equivalence on $L$ that collapses only the bottom elements of these lattices, as well as their top elements: $\varepsilon =eq(\{\{0^{L_i}\ |\ i\in I\},\{1^{L_i}\ |\ i\in I\}\}\cup \{\{x\}\ |\ x\in L\setminus \{0^{L_i},1^{L_i}\ |\ i\in I\}\})\in {\rm Eq}(L)$; denote by $0=0^{L_i}/\varepsilon $ and $1=1^{L_i}/\varepsilon $ for some $i\in I$; then, for every $i\in I$, $\varepsilon \cap L_i^2=\Delta _{L_i}\in {\rm Con}(L_i)$, so $L_i\cong L_i/\varepsilon $; we identify each $L_i$ with $L_i/\varepsilon $, by identifying $x$ with $x/\varepsilon $ for all $x\in L$, thus obtaining $0=0^{L_i}$ and $1=1^{L_i}$ for all $i\in I$; now we set $\boxplus _{i\in I}L_i=L/\varepsilon $ and $\displaystyle \leq =\bigcup _{i\in I}\leq ^{L_i}$. In particular, we denote by ${\cal M}_{|I|}=\boxplus _{i\in I}{\cal L}_3$ the modular lattice of length $3$ and cardinality $|I|+2$, which is clearly simple. If $(L_i,\cdot ^{\prime i})_{i\in I}$ is a nonempty family of nontrivial bi--lattices, then the {\em horizontal sum} of this family is the bi--lattice $(\boxplus _{i\in I}L_i,\cdot ^{\prime })$, whose underlying bounded lattice is the horizontal sum of the family of the bounded lattice reducts $(L_i)_{i\in I}$ and whose involution is defined by: $\cdot ^{\prime }\mid _{L_i}=\cdot ^{\prime i}$ for all $i\in I$.

Let $(L,\leq ^L)$ be a lattice with top element $1^L$ and $(M,\leq ^M)$ a lattice with bottom element $0^M$. Recall that the {\em ordinal sum} of $L$ with $M$ is the lattice $(L\oplus M,\leq )$ obtained by glueing the top element of $L$ and the bottom element of $M$ together, thus stacking $M$ on top of $L$. More precisely, if we denote by $\varepsilon $ the equivalence on $L\amalg M$ that only collapses $1^L$ with $0^M$: $\varepsilon =eq(\{\{1^L,0^M\}\}\}\cup \{\{x\}\ |\ x\in L\amalg M\setminus \{1^L,0^M\}\})\in {\rm Eq}(L\amalg M)$, then, noting that $\varepsilon \cap L^2=\Delta _L\in {\rm Con}(L)$ and $\varepsilon \cap M^2=\Delta _M\in {\rm Con}(M)$, we identify $L$ with $L/\varepsilon \cong L$ and $M$ with $M/\varepsilon \cong M$ by identifying each $x\in L\amalg M$ with $x/\varepsilon $; now we let $L\oplus M=(L\amalg M)/\varepsilon $ and $\displaystyle \leq =\leq ^L\cup \leq ^M\cup \{(x,y)\ |\ x\in L,y\in M\}$.

If, for every $\alpha \in {\rm Con}(L)$ and every $\beta \in {\rm Con}(M)$, we denote by $\alpha \oplus \beta $ the equivalence on $L\oplus M$ whose classes are those of the equivalences $\alpha $ and $\beta $, excepting $1^L/\alpha $ and $0^M/\beta $, along with the union of the classes $1^L/\alpha $ and $1^L/\beta =0^M/\beta $: $\alpha \oplus \beta =eq((L/\alpha \setminus 1^L/\alpha )\cup (M/\beta \setminus 0^M/\beta )\cup \{1^L/\alpha \cup 0^M/\beta \})$, then, clearly, $\alpha \oplus \beta \in {\rm Con}(L\oplus M)$. Furthermore, since $L$ and $M$ are sublattices of $L\oplus M$, for every $\theta \in {\rm Con}(L\oplus M)$, we have $\theta \cap L^2\in {\rm Con}(L)$ and $\theta \cap M^2\in {\rm Con}(M)$, and clearly $\theta =(\theta \cap L^2)\oplus (\theta \cap M^2)$. Therefore the map $(\alpha ,\beta )\mapsto \alpha \oplus \beta $ is a lattice isomorphism from ${\rm Con}(L)\times {\rm Con}(M)$ to ${\rm Con}(L\oplus M)$.

Clearly, the ordinal sum of bounded lattices is associative and so is the operation $\oplus $ on congruences of those bounded lattices.

Let us also note that ${\rm Filt}(L\oplus M)={\rm Filt}(M)\cup \{F\cup L\ |\ F\in {\rm Filt}(L)\}$ and ${\rm Id}(L\oplus M)={\rm Id}(L)\cup \{L\cup I\ |\ I\in {\rm Id}(M)\}$, thus $|{\rm Filt}(L\oplus M)|=|{\rm Filt}(L)|+|{\rm Filt}(M)|-1$ and $|{\rm Id}(L\oplus M)|=|{\rm Id}(L)|+|{\rm Id}(M)|-1$, where we let $\kappa -1=\kappa $ for any infinite cardinal number $\kappa $.

Let $L$ be a lattice with top element, $f:L\rightarrow L^d$ a dual lattice isomorphism and $(K,\cdot ^{\prime K})$ a bi--lattice. Then $L\oplus K\oplus L^d$, and in particular $L\oplus L^d$ in the case when $K$ is the one--element chain, becomes an i--lattice with the involution $\cdot ^{\prime }:L\oplus K\oplus L^d\rightarrow L\oplus K\oplus L^d$ defined by: $\cdot ^{\prime }\mid _L=f$, $\cdot ^{\prime }\mid _K=\cdot ^{\prime K}$ and $\cdot ^{\prime }\mid _{L^d}=f^{-1}$. Notice that, if $K$ is a pseudo--Kleene algebra, then $L\oplus K\oplus L^d$ satisfies \textcircled{k}, thus $L\oplus K\oplus L^d$ is a pseudo--Kleene algebra if $L$ is a bounded lattice. In particular, $L\oplus L^d$ is a pseudo--Kleene algebra for any bounded lattice $L$.

For any i--lattice $(A,\cdot ^{\prime })$, if we denote by $U^{\prime }=\{(a^{\prime },b^{\prime })\ |\ (a,b)\in U\}$ for all $U\subseteq A^2$, then we clearly have ${\rm Con}_{\I }(A)=\{\theta \in {\rm Con}(A)\ |\ \theta =\theta ^{\prime }\}$, from which it is immediate that, with the notations above, ${\rm Con}_{\I }(L\oplus K\oplus L^d)=\{\alpha \oplus \beta \oplus \alpha ^{\prime }\ |\ \alpha \in {\rm Con}(L),\beta \in {\rm Con}_{\I }(K)\}\cong {\rm Con}(L)\times {\rm Con}_{\I }(K)$, in particular ${\rm Con}_{\I }(L\oplus L^d)=\{\alpha \oplus \alpha ^{\prime }\ |\ \alpha \in {\rm Con}(L)\}\cong {\rm Con}(L)$; see also \cite{rgcmfp,kumu,eunoucard,eucardbi}. Note, also, that, for any bi--lattice $A$, ${\rm Con}_{\I 01}(A)={\rm Con}_{\I 0}(A)$.

It is straightforward that, if $L$ is a non--trivial bounded lattice and $K\in \KL $, then the pseudo--Kleene algebra $L\oplus K\oplus L^d$ becomes an antiortholattice when endowed with the trivial Brouwer complement \cite{pbzsums}. It is routine to prove that, for any antiortholattice $A$, ${\rm Con}_{\BZ }(A)={\rm Con}_{\BZ 0}(A)\cup \{\nabla _A\}={\rm Con}_{\I 0}(A)\cup \{\nabla _A\}\cong {\rm Con}_{\I 0}(A)\oplus {\cal L}_2$; see also \cite{PBZ2,pbzsums}. Therefore, if $L$ is a non--trivial bounded lattice and $K\in \KL $, then ${\rm Con}_{\BZ }(L\oplus K\oplus L)={\rm Con}_{\I 0}(L\oplus K\oplus L)\cup \{\nabla _{L\oplus K\oplus L}\}=\{\alpha \oplus \beta \oplus \alpha ^{\prime}\ |\ \alpha \in {\rm Con}_0(L),\beta \in {\rm Con}_{\I }(K)\}\}\cup \{\nabla _{L\oplus K\oplus L}\}\cong ({\rm Con}_0(L)\times {\rm Con}_{\I }(K))\oplus {\cal L}_2$, thus, if $L$ is $0$--regular, so that ${\rm Con}_0(L)=\{\Delta _L\}$, then ${\rm Con}_{\BZ }(L\oplus K\oplus L)=\{\Delta _L\oplus \beta \oplus \Delta _L\ |\ \beta \in {\rm Con}_{\I }(K)\}\cup \{\nabla _{L\oplus K\oplus L}\}=\{eq(K/\beta \cup \{\{0\},\{1\}\})\ |\ \beta \in {\rm Con}_{\I }(K)\}\cup \{\nabla _{L\oplus K\oplus L}\}\cong {\rm Con}_{\I }(K)\oplus {\cal L}_2$, thus $|{\rm Con}_{\BZ }(L\oplus K\oplus L)|=|{\rm Con}_{\I }(K)|+1$; in particular, ${\rm Con}_{\BZ }({\cal L}_2\oplus K\oplus {\cal L}_2)=\{eq(K/\beta \cup \{\{0\},\{1\}\})\ |\ \beta \in {\rm Con}_{\I }(K)\}\cup \{\nabla _{{\cal L}_2\oplus K\oplus {\cal L}_2}\}\cong {\rm Con}_{\I }(K)\oplus {\cal L}_2$, thus $|{\rm Con}_{\BZ }({\cal L}_2\oplus K\oplus {\cal L}_2)|=|{\rm Con}_{\I }(K)|+1$, and, for any $0$--regular nontrivial bounded lattice $L$, the antiortholattice $L\oplus L^d$ is simple; see also \cite{pbzsums,eucardbi}.

See in \cite{pbzsums,eunoucard,eucardbi} the congruences of any horizontal sum of nontrivial bounded (involution) lattices. Now let us look at the particular case of the horizontal sum of a bounded (involution) lattice $L$ having $|L|>2$ with the four--element Boolean algebra or with two copies of the three--element chain; the following hold for any of these two possible definitions of the involution on the lattice ${\cal L}_2^2$, whose incomparable elements we will denote by $a$ and $b$. For any $\theta \in {\rm Con}(L\boxplus {\cal L}_2^2)\setminus \{\nabla _{L\boxplus {\cal L}_2^2}\}$ and any $x\in L\setminus \{0,1\}$, the simple lattice $S=\{0,a,x,b,1\}\cong {\cal M}_3$ is a sublattice of $L\boxplus {\cal L}_2^2$, so $\theta \cap S^2\in {\rm Con}(S)$, hence $\theta \cap S^2=\Delta _S$ since $(0,1)\notin \theta $, thus neither of the elements $a$, $x$, $b$ belongs to $0/\theta $ or $1/\theta $. On the other hand, for any $\alpha \in {\rm Con}_{01}(L)$, the equivalence $eq(L/\alpha \cup \{\{a\},\{b\}\})$ obtained by putting the classes of $\alpha $ together with those of the only member $\Delta _{{\cal L}_2^2}$ of ${\rm Con}_{01}({\cal L}_2^2)$ is a lattice congruence of $L\boxplus {\cal L}_2^2$. These properties and the fact that $L$ and ${\cal L}_2^2$ are subalgebras of $L\boxplus {\cal L}_2^2$, so that, for any (involution--preserving) lattice congruence $\theta $ of $L\boxplus {\cal L}_2^2$, $\theta \cap L^2$ and $\theta \cap {\cal L}_2^2$ are (involution--preserving) lattice congruences of $L$ and ${\cal L}_2^2$, respectively, show that ${\rm Con}(L\boxplus {\cal L}_2^2)={\rm Con}_{01}(L\boxplus {\cal L}_2^2)\cup \{\nabla _{L\boxplus {\cal L}_2^2}\}=\{eq(L/\alpha \cup \{\{a\},\{b\}\})\ |\ \alpha \in {\rm Con}_{01}(L)\}\cup \{\nabla _{L\boxplus {\cal L}_2^2}\}\cong {\rm Con}_{01}(L)\oplus {\cal L}_2$ and, when $L\in \BI $, ${\rm Con}_{\I }(L\boxplus {\cal L}_2^2)={\rm Con}_{\I 0}(L\boxplus {\cal L}_2^2)\cup \{\nabla _{L\boxplus {\cal L}_2^2}\}=\{eq(L/\alpha \cup \{\{a\},\{b\}\})\ |\ \alpha \in {\rm Con}_{\I 0}(L)\}\cup \{\nabla _{L\boxplus {\cal L}_2^2}\}\cong {\rm Con}_{\I 0}(L)\oplus {\cal L}_2$.

We will denote by $(\B (L),\leq )$ the bounded lattice obtained from a lattice $(L,\leq ^L)$ by adding a new top element $1$ and a new bottom element $0$: $\B (L)=L\amalg \{0\}\amalg \{1\}$ and $\leq =\leq ^L\cup \{(0,x),(x,1)\ |\ x\in \B (L)\}\setminus \{(1,0)\}$. We have $|\B (L)|=|L|+2$ and ${\rm Filt}(\B (L))=\{\{1\},\B (L)\}\cup \{F\cup \{1\}\ |\ F\in {\rm Filt}(L)\}$, ${\rm Id}(\B (L))=\{\{0\},\B (L)\}\cup \{I\cup \{0\}\ |\ I\in {\rm Id}(L)\}$, so that $|{\rm Filt}(\B (L))|=|{\rm Filt}(L)|+2$ and $|{\rm Id}(\B (L))|=|{\rm Id}(L)|+2$. Note that, in $\B (L)$, $0$ is meet--irreducible and $1$ is join--irreducible, from which it is immediate that $eq(\{0\},L,\{1\})\in {\rm Con}(\B (L))$ and, moreover, $eq(L/\theta \cup \{\{0\},\{1\}\})\in {\rm Con}(\B (L))$ for all $\theta \in {\rm Con}(L)$ and, since $L$ is a sublattice of $\B (L)$, for any $\alpha \in {\rm Con}(\B (L))$, we have $\alpha \cap L^2\in {\rm Con}(L)$, therefore ${\rm Con}_{01}(\B (L))=\{eq(L/\theta \cup \{\{0\},\{1\}\})\ |\ \theta \in {\rm Con}(L)\}$. If $(L,\cdot ^{\prime L})$ is an involution lattice, then $(\B (L),\cdot ^{\prime })$ is a bounded involution lattice with $\cdot ^{\prime }\mid _L=\cdot ^{\prime L}$, $0^{\prime }=1$ and $1^{\prime }=0$ and, by the above, ${\rm Con}_{\I 0}(\B (L))=\{eq(L/\theta \cup \{\{0\},\{1\}\})\ |\ \theta \in {\rm Con}_{\I }(L)\}$.

Now let us look at a construction we have used in \cite{gccm,eunoucard,eucardbi}: let $M$ be a lattice and let us consider the bounded lattice $L=\B (M)\boxplus {\cal L}_2^2$, with the incomparable elements of ${\cal L}_2^2$ denoted $a$ and $b$. Then $|L|=|M|+4$, ${\rm Filt}(L)={\rm Filt}(\B (M))\cup \{\{a,1\},\{b,1\}\}$ and ${\rm Id}(L)={\rm Id}(\B (M))\cup \{\{0,a\},\{0,b\}\}$, thus $|{\rm Filt}(L)|=|{\rm Filt}(M)|+4$ and $|{\rm Id}(L)|=|{\rm Id}(M)|+4$.

By the above, if $M$ is an i--lattice, then $\B (M)$ becomes a bi--lattice, hence $L$ can be organized as a bi--lattice either as the horizontal sum of bi--lattices $L=\B (M)\boxplus {\cal L}_2^2$ of the bi--lattice $\B (M)$ with the four--element Boolean algebra or as the horizontal sum of bi--lattices $L={\cal L}_3\boxplus \B (M)\boxplus {\cal L}_3$ of the bi--lattice $\B (M)$ with two copies of the three--element involution chain. The first of these horizontal sums of bi--lattices satisfies the property that, in the particular case when $M$ satisfies \textcircled{k} for all $a,b\in M$, which means that $M$ is a pseudo--Kleene algebra in the particular case when $M$ is a bi--lattice, then $L=\B (M)\boxplus {\cal L}_2^2$ is a pseudo--Kleene algebra. The following hold for either of these definitions of the involution of $L$:\vspace*{3pt}

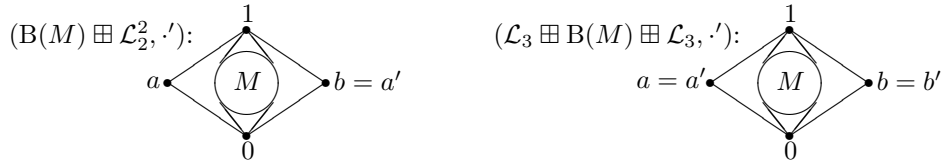
\begin{center}\begin{tabular}{cc}\begin{picture}(40,40)(0,0)
\put(-70,35){$(\B (M)\boxplus {\cal L}_2^2,\cdot ^{\prime })$:}
\put(20,0){\circle*{3}}
\put(-10,20){\circle*{3}}
\put(50,20){\circle*{3}}
\put(20,40){\circle*{3}}
\put(18,-9){$0$}
\put(-18,18){$a$}
\put(53,17){$b=a^{\prime }$}
\put(18,43){$1$}
\put(20,20){\circle{22}}
\put(20,0){\line(-4,5){10}}
\put(20,0){\line(4,5){10}}
\put(15,17){$M$}
\put(20,40){\line(-4,-5){10}}
\put(20,40){\line(4,-5){10}}
\put(20,0){\line(-3,2){30}}
\put(20,0){\line(3,2){30}}
\put(20,40){\line(-3,-2){30}}
\put(20,40){\line(3,-2){30}}\end{picture}
&\hspace*{150pt}
\begin{picture}(40,40)(0,0)
\put(-92,35){$({\cal L}_3\boxplus \B (M)\boxplus {\cal L}_3,\cdot ^{\prime })$:}
\put(20,0){\circle*{3}}
\put(-10,20){\circle*{3}}
\put(50,20){\circle*{3}}
\put(20,40){\circle*{3}}
\put(18,-9){$0$}
\put(-38,18){$a=a^{\prime }$}
\put(53,17){$b=b^{\prime }$}
\put(18,43){$1$}
\put(20,20){\circle{22}}
\put(20,0){\line(-4,5){10}}
\put(20,0){\line(4,5){10}}
\put(15,17){$M$}
\put(20,40){\line(-4,-5){10}}
\put(20,40){\line(4,-5){10}}
\put(20,0){\line(-3,2){30}}
\put(20,0){\line(3,2){30}}
\put(20,40){\line(-3,-2){30}}
\put(20,40){\line(3,-2){30}}\end{picture}
\end{tabular}\end{center}

By the above, ${\rm Con}(L)=\{eq(M/\theta \cup \{\{0\},\{a\},\{b\},\{1\}\})\ |\ \theta \in {\rm Con}(M)\}\cup \{\nabla _L\}\cong {\rm Con}(M)\oplus {\cal L}_2$ and, if $M\in \I $, then ${\rm Con}_{\I }(L)=\{eq(M/\theta \cup \{\{0\},\{a\},\{b\},\{1\}\})\ |\ \theta \in {\rm Con}_{\I }(M)\}\cup \{\nabla _L\}\cong {\rm Con}_{\I }(M)\oplus {\cal L}_2$, so $|{\rm Con}(L)|=|{\rm Con}(M)|+1$ and, if $M\in \I $, then $|{\rm Con}_{\I }(L)|=|{\rm Con}_{\I }(M)|+1$.

\section{The Theorems}

Let $L$ be a lattice. Since the partitions of $L$ are among the subsets of ${\cal P}(L)$ of cardinality at most $|L|$, we have $|{\rm Con}(L)|\leq |{\rm Eq}(L)|=|{\rm Part}(L)|\leq (2^{|L|})^{|L|}=2^{|L|\cdot |L|}$, so that $|{\rm Con}(L)|\leq |{\rm Eq}(L)|\leq 2^{|L|}$ if $L$ is infinite. Actually, by \cite{gcze,free}, if $L$ is finite, then $|{\rm Con}(L)|\leq 2^{|L|-1}$, so $L$ has at most as many congruences as subsets regardless of whether it is infinite.

Concerning the numbers of filters and ideals of $L$, we have $|L|=|{\rm PFilt}(L)|=|{\rm PId}(L)|\leq |{\rm Filt}(L)|,|{\rm Id}(L)|\leq |{\cal P}(L)|=2^{|L|}$. Thus, under the Generalized Continuum Hypothesis, if $L$ is infinite, then $|{\rm Filt}(L)|,|{\rm Id}(L)|\in \{|L|,2^{|L|}\}$, and, if, for some infinite cardinal number $\nu $, $L$ has $\{|{\rm Filt}(L)|,|{\rm Id}(L)|\}=\{\nu ,2^{\nu }\}$, then $|L|=\nu $.

Throughout the rest of this paper, $\V $ will be an arbitrary variety of algebras of the same similarity type.

Let $(I,\leq )$ be an ordered set and $(A_{\mu })_{\mu \in I}$ a family of members of $\V $. Recall that $(A_{\mu })_{\mu \in I}$ is called a {\em directed system} of members of $\V $ iff it satisfies the following condition, stating that each $A_{\lambda }$ is a proper subalgebra of every $A_{\mu }$ with $\lambda <\mu $:

\begin{flushleft}\begin{tabular}{cl}
\textcircled{s}$_{\V }$ & for all $\lambda ,\mu \in I$ with $\lambda <\mu $, $A_{\lambda }\subseteq A_{\mu }$ and $A_{\lambda }\in \S _{\V }(A_{\mu })\setminus \{A_{\mu }\}$\end{tabular}\end{flushleft}

If $(A_{\mu })_{\mu \in I}$ is a directed system, then we can define the {\em directed union} of $(A_{\mu })_{\mu \in I}$ to be the member $A$ of $\V $ with $\displaystyle A=\bigcup _{\mu \in I}A_{\mu }$ and, for every $\star $ belonging to the signature of $\V $ and every $\mu \in I$, $\star ^A\mid _{A_{\mu }}=\star ^{A_{\mu }}$.

Note that, if $(A_{\mu })_{\mu \in I}$ is a directed system, then so is $(A_{\mu })_{\mu \in J}$ for every subset $J$ of $I$, thus, trivially, for each $\mu \in I$, $A_{\mu }$ is the directed union of the family $(A_{\lambda })_{\lambda \in I,\lambda \leq \mu }$.

The following condition states that each nontrivial congruence of every $A_{\mu }$ has an $A_{\lambda }$ with $\lambda <\mu $ as unique nonsingleton congruence class, and such a congruence of $A_{\mu }$ exists for every $\lambda <\mu $:

\begin{flushleft}\begin{tabular}{cl}
\textcircled{c}$_{\V }$ & for all $\mu \in I$, ${\rm Con}_{\V }(A_{\mu })=\{\Delta _{A_{\mu }},\nabla _{A_{\mu }}\}\cup \{eq(\{A_{\lambda }\}\cup \{\{x\}\ |\ x\in A_{\mu }\setminus A_{\lambda }\})\ |\ \lambda \in I,\lambda <\mu \}$\end{tabular}\end{flushleft}

If $\V $ is the variety of lattices or that of bounded lattices, then we denote the conditions \textcircled{s}$_{\V }$ and \textcircled{c}$_{\V }$, simply, by \textcircled{s} and \textcircled{c}, respectively.

For the calculations that follow in this section, note that an equivalent form of \textcircled{c}$_{\V }$ is: for all $\mu \in I$, ${\rm Con}_{\V }(A_{\mu })=\{\Delta _{A_{\mu }}\}\cup \{eq(\{A_{\lambda }\}\cup \{\{x\}\ |\ x\in A_{\mu }\setminus A_{\lambda }\})\ |\ \lambda \in I,\lambda \leq \mu \}$. Note, also, that a singleton family satisfies condition \textcircled{c}$_{\V }$ iff its member is a simple algebra from $\V $.

\begin{lemma}{\rm \cite[Lemma $3.2$]{gccm}} Let $\iota $ be a limit ordinal, $\sigma $ an ordinal with $\sigma <\iota $, $I=\{\mu \ |\ \sigma \leq \mu <\iota \}$, $(A_{\mu })_{\mu \in I}$ a family of members of $\V $ and $A_{\iota }$ the directed union of the family $(A_{\mu })_{\mu \in I}$. If the family $(A_{\mu })_{\mu \in I}$ satisfies conditions \textcircled{s}$_{\V }$ and \textcircled{c}$_{\V }$, then the family $(A_{\mu })_{\mu \in I\cup \{\iota \}}$ also satisfies conditions \textcircled{s}$_{\V }$ and \textcircled{c}$_{\V }$.\label{czedlis}\end{lemma}

\begin{lemma} Let $\tau $ be an ordinal, $I=\{\mu \ |\ 2\leq \mu \leq \tau \}$ and $(A_{\mu })_{\mu \in I}$ a family of members of $\V $ that satisfies conditions \textcircled{s}$_{\V }$ and \textcircled{c}$_{\V }$ and in which $A_2$ is a nontrivial algebra. Then:\begin{enumerate}
\item\label{elemncg1} Then $|{\rm Con}_{\V }(A_{\mu })|=|\mu |$ for all $\mu \in I$.
\item\label{elemncg2} Assume that $A_2$ is infinite and has $|A_2|=\nu \geq |\tau |$, that, for each ordinal $\mu \in I$ such that $\mu +1\in I$, we have $|A_{\mu +1}|=|A_{\mu }|+\kappa _{\mu }$ for some cardinal number $\kappa _{\mu }\leq \nu $, and that, for each limit ordinal $\iota \in I$, $A_{\iota }$ is the directed union of the family $(A_{\lambda })_{\lambda \in I,\lambda <\iota }$. Then $|A_{\lambda }|=\nu $ for each $\lambda \in I$.\end{enumerate}\label{elemncg}\end{lemma}

\begin{proof} (\ref{elemncg1}) Let $\mu \in I$. From the fact that $A_2$ is nontrivial and condition \textcircled{s}$_{\V }$, that ensures us, in particular, that $A_{\mu }$ is nontrivial, as well, we get that, for all $\eta ,\lambda \in I$ with $\eta ,\lambda <\mu $ and $\eta \neq \lambda $, the congruences $\Delta _{A_{\mu }}$, $eq(\{A_{\eta }\}\cup \{\{x\}\ |\ x\in A_{\mu }\setminus A_{\eta }\})$, $eq(\{A_{\lambda }\}\cup \{\{x\}\ |\ x\in A_{\mu }\setminus A_{\lambda }\})$ and $\nabla _{A_{\mu }}$ are pairwise distinct. Therefore $|{\rm Con}_{\V }(A_{\mu })|=2+|\mu |-2=|\mu |$.

\noindent (\ref{elemncg2}) We apply induction. By the hypothesis, $|A_2|=\nu $.

Now let $\iota \in I\setminus \{2\}$, which means that $\iota $ is an ordinal with $3\leq \iota \leq \tau $.

If $\iota $ is a successor ordinal, $\iota =\mu +1$ for a (unique) ordinal $\mu $ with $2\leq \mu <\tau $ and such that $|A_{\mu }|=\nu $, then  $|A_{\iota }|=|A_{\mu }|+\kappa _{\mu }=\nu +\kappa _{\mu }=\nu $ since $\nu \geq \kappa _{\mu }$.

If $\iota $ is a limit ordinal such that, for all ordinals $\mu $ with $2\leq \mu <\iota $, $|A_{\mu }|=\nu $, then $\displaystyle \nu =|A_2|\leq |A_{\iota }|\leq \sum _{2\leq \lambda <\iota }|A_{\mu }|=\sum _{2\leq \lambda <\iota }\nu \leq |\iota |\cdot \nu \leq |\tau |\cdot \nu =\nu $ since $\nu \geq |\tau |$.\end{proof}

\begin{lemma}{\rm \cite[Lemma $3.1$]{gccm}} Let $I$ be an ideal of a lattice $K$ such that $K$ satisfies the following condition:

\begin{tabular}{cl}
\textcircled{g}$_{I}$ & for all $(x_n)_{n\in \N }\subseteq K$, if $x_n>x_{n+1}$ for all $n\in \N $,\\
& then $x_n\in I$ for all but finitely many $n\in \N $.\end{tabular}

Then every nonprincipal filter of $K$ is generated by a filter of $I$.\label{alsoczedlis}\end{lemma}

\begin{theorem} {\rm \cite[Theorem 1.1]{gccm}} For any infinite cardinal number $\nu $ and any cardinal number $\kappa $ with $2\leq \kappa \leq \nu $ or $\kappa =2^{\nu }$, there exists a bounded lattice $M_{\nu ,\kappa }$ with $|M_{\nu ,\kappa }|=|{\rm Filt}(M_{\nu ,\kappa })|=\nu $, $|{\rm Id}(M_{\nu ,\kappa })|=2^{\nu }$ and $|{\rm Con}(M_{\nu ,\kappa })|=\kappa $. Furthermore, $M_{\nu ,2^{\nu }}$ can be chosen to be distributive.\label{gccmth}\end{theorem}

\begin{remark} Let $\nu $ be an infinite cardinal number. For every $\kappa $ as in Theorem \ref{gccmth}, with the notations from this theorem, we have ${\rm Con}(M_{\nu ,\kappa }^d)={\rm Con}(M_{\nu ,\kappa })$, ${\rm Filt}(M_{\nu ,\kappa }^d)={\rm Id}(M_{\nu ,\kappa })$ and ${\rm Id}(M_{\nu ,\kappa }^d)={\rm Filt}(M_{\nu ,\kappa })$. If we consider, for each such $\kappa $, the pseudo--Kleene algebra $L_{\nu ,\kappa }=M_{\nu ,\kappa }\oplus M_{\nu ,\kappa }^d$, then $|{\rm Con}_{\I }(L_{\nu ,\kappa })|=|{\rm Con}(M_{\nu ,\kappa })|=\kappa $ and $|{\rm Filt}(L_{\nu ,\kappa })|=|{\rm Id}(L_{\nu ,\kappa })|=|{\rm Filt}(M_{\nu ,\kappa })|+|{\rm Id}(M_{\nu ,\kappa })|-1=\nu +2^{\nu }-1=2^{\nu }$. By the above, $L_{\nu ,2^{\nu }}$ can be chosen to be a Kleene lattice.

Under the GCD, the cardinal numbers $\kappa $ above take each value between $2$ and the cardinality $2^{\nu }$ of the sets of subsets of the lattices $M_{\nu ,\kappa }$ and $M_{\nu ,\kappa }^d$, which have the only possible values for the cardinalities of their sets of filters and ideals under the condition that these cardinalities are different.\label{fromgccmth}\end{remark}

Now we revisit the technique from the proof in \cite{gccm} of Theorem \ref{gccmth}; we apply the construction from \cite[Section $3$]{gccm} to an arbitrary lattice $L$ and we also consider the case when $L$ is an involution lattice. We will apply Lemma \ref{alsoczedlis} in a slightly different manner than in \cite[Section $3$]{gccm}, so that, in statement (\ref{thelatconstr2}) of the following proposition, we do not need to confine ourselves to the cases when $L$ satisfies the Descending or the Ascending Chain Condition or has as many filters or ideals as subsets; instead, this statement holds for any lattice $L$ and does not necessitate enforcing the Continuum Hypothesis. Of course, if a lattice $L$ has all filters principal, then $|{\rm Filt}(L)|=|L|$, and the same goes for ideals, but the converses do not hold; for instance, if $\nu $ is an infinite cardinal number and we let $L={\cal M}_{2^{\nu }}\oplus {\cal L}_2^{\nu }$, then, since the Boolean algebra ${\cal L}_2^{\nu }$ has as many filters and as many ideals and subsets, while the lattice ${\cal M}_{2^{\nu }}$ has finite length and thus all filters and ideals principal, we have $|{\rm Filt}(L)|=|{\rm Id}(L)|=|{\rm Filt}({\cal M}_{2^{\nu }})|+|{\rm Filt}({\cal L}_2^{\nu })|-1=|{\rm Id}({\cal M}_{2^{\nu }})|+|{\rm Id}({\cal L}_2^{\nu })|-1=2^{\nu }+2^{\nu }-1=2^{\nu }=|L|$, and $L$ has nonprincipal filters, namely the nonprincipal filters of ${\cal L}_2^{\nu }$, and nonprincipal ideals, namely the unions of ${\cal M}_{2^{\nu }}$ with nonprincipal ideals of ${\cal L}_2^{\nu }$; see in \cite{eucard} more examples of lattices with as many filters and ideals as elements, but having nonprincipal filters and nonprincipal ideals, thus failing both the Descending and the Ascending Chain Condition.

Let $L$ be an (involution) lattice, $\kappa $ be a cardinal number with $2\leq \kappa $, $\sigma $ an ordinal with $|\sigma |=\kappa $, $\tau =\sigma +1$, so that $|\tau |=|\sigma |+1=\kappa +1$, and $I=\{\mu \ |\ 2\leq \mu \leq \tau \}=\{\mu \ |\ 2\leq \mu \leq \sigma \}\cup \{\tau \}$. We add the succesor ordinal $\tau $ to ensure the boundeness of the lattice $L_{\max (I)}=L_{\tau }$ in the family of lattices we construct in what follows.

We define inductively a family $(L_{\mu })_{\mu \in I}$ of (involution) lattices, in the following way: $L_2=L$ and, for every $\iota \in I\setminus \{2\}=\{\mu \ |\ 3\leq \mu \leq \tau \}$:\begin{itemize}
\item if $\iota $ is a successor ordinal, $\iota =\mu +1$ for a $\mu \in I$, then we define $L_{\iota }=\B (L_{\mu })\boxplus {\cal L}_2^2$, as a horizontal sum of bounded (involution) lattices;
\item if $\iota $ is a limit ordinal, then we define $L_{\iota }$ to be the directed union of the family of (involution) lattices $(L_{\mu })_{2\leq \mu <\iota }$.\end{itemize}

Note that $L_{\iota }$ is a bounded (involution) lattice for every successor ordinal $\iota \in I\setminus \{2\}$, in particular $L_{\tau }$ is a bounded (involution) lattice. Moreover, in the subfamily $(L_{\mu })_{\mu \in I\setminus \{2\}}$, the bounded members are exactly those indexed by successor ordinals.

\begin{remark} For the case of bounded involution lattices in the following proposition, we can also define, for every $\mu \in I$ such that $\iota =\mu +1\in I$, the bi--lattice $L_{\iota }$ to be the horizontal sum of bi--lattices ${\cal L}_3\boxplus \B (L_{\mu })\boxplus {\cal L}_3$, and the statements in the proposition still hold.

However, if we let the involution of $L_{\iota }$ to be defined as in the horizontal sum of the bi--lattice $\B (L_{\mu })$ with the four--element Boolean algebra, and $L$ satisfies condition \textcircled{k} (so that $L$ is a pseudo--Kleene algebra in the particular case when $L$ is bounded), then it is easy to see that each member of the family $(L_{\mu })_{\mu \in I}$ satisfies \textcircled{k}, and thus $L_{\iota }$ is a pseudo--Kleene algebra for every successor ordinal $\iota \in I\setminus \{2\}$, in particular $L_{\tau }$ is a pseudo--Kleene algebra.\label{pkas}\end{remark}

\begin{proposition} With the notations above, if $L$ is non--trivial, then:\begin{enumerate}
\item\label{thelatconstr1} the family $(L_{\mu })_{\mu \in I}$ satisfies condition \textcircled{s} and, if $L\in \I $, so that $(L_{\mu })_{\mu \in I}\subset \I $, then also condition \textcircled{s}$_{\I }$;
\item\label{thelatconstr3} if $L$ is a nontrivial simple lattice, then the family $(L_{\mu })_{\mu \in I}$ satisfies condition \textcircled{c}, so, for all $\mu \in I$, $|{\rm Con}(L_{\mu })|=|\mu |$, in particular $|{\rm Con}(L_{\tau })|=\kappa $;
\item\label{thelatconstr4} if $L$ is a nontrivial simple i--lattice, then the family $(L_{\mu })_{\mu \in I}$ satisfies condition \textcircled{c}$_{\I }$, so, for all $\mu \in I$, $|{\rm Con}_{\I }(L_{\mu })|=|\mu |$, in particular $|{\rm Con}_{\I }(L_{\tau })|=\kappa $;
\item\label{thelatconstr5} if $L$ is a nontrivial i--lattice with a simple lattice reduct, then the family $(L_{\mu })_{\mu \in I}$ satisfies conditions \textcircled{c} and \textcircled{c}$_{\I }$, so, for all $\mu \in I$, ${\rm Con}_{\I }(L_{\mu })={\rm Con}(L_{\mu })$ and $|{\rm Con}_{\I }(L_{\mu })|=|{\rm Con}(L_{\mu })|=|\mu |$, in particular $|{\rm Con}_{\I }(L_{\tau })|=|{\rm Con}(L_{\tau })|=\kappa $;
\item\label{thelatconstr2} if $L$ is an infinite lattice and $|L|=\nu \geq \kappa $, then, for all $\mu \in I$, $|L_{\mu }|=\nu $, $|{\rm Filt}(L_{\mu })|=|{\rm Filt}(L)|$ and $|{\rm Id}(L_{\mu })|=|{\rm Id}(L)|$.\end{enumerate}\label{thelatconstr}\end{proposition}

\begin{proof} (\ref{thelatconstr1}) The singleton family $\{L_2\}=\{L\}$ trivially satisfies condition \textcircled{s}, respectively \textcircled{s}$_{\I }$. For every $\iota \in I\setminus \{1\}$, if $\iota $ is a successor ordinal, $\iota =\mu +1$ for some $\mu \in I$, then, by the definition of $L_{\iota }$, we have $L_{\mu }\in \S (L_{\iota })$, respectively $L_{\mu }\in \S _{\I }(L_{\iota })$, while, if $\iota $ is a limit ordinal, then, again by the definition of $L_{\iota }$, we have $L_{\lambda }\in \S (L_{\iota })$, respectively $L_{\lambda }\in \S _{\I }(L_{\iota })$, for each $2\leq \lambda <\iota $. So an immediate induction argument shows that the family $(L_{\mu })_{\mu \in I}$ satisfies condition \textcircled{s}, respectively \textcircled{s}$_{\I }$.

\noindent (\ref{thelatconstr3}),(\ref{thelatconstr4}) We apply induction. Assume that $L_2=L$ is a simple lattice, respectively a simple i--lattice, so that the singleton family $\{L_2\}=\{L\}$ satisfies condition \textcircled{c}, respectively \textcircled{c}$_{\I }$. Now let $\iota \in I\setminus \{2\}$, and let $\V $ be the variety of lattices in the case of (\ref{thelatconstr3}), respectively $\V =\I $ in the case of (\ref{thelatconstr4}).

If $\iota $ is a successor ordinal, $\iota =\mu +1$ for some $\mu \in I$ such that the family $(L_{\lambda })_{2\leq \lambda \leq \mu }$ satisfies condition \textcircled{c}$_{\V }$, then $L_{\iota }=\B (L_{\mu })\boxplus {\cal L}_2^2$ and ${\rm Con}_{\V }(L_{\mu })=\{\Delta _{L_{\mu }}\}\cup \{eq(\{L_{\lambda }\}\cup \{\{x\}\ |\ x\in L_{\mu }\setminus L_{\lambda }\})\ |\ 2\leq \lambda \leq \mu \}$, therefore, by (\ref{thelatconstr1}) and the congruences of the construction $L_{\iota }=\B (L_{\mu })\boxplus {\cal L}_2^2$ determined at the end of section \ref{thealg}, ${\rm Con}_{\V }(L_{\iota })=\{\nabla _{L_{\iota }}\}\cup \{eq(L_{\mu }/\theta \cup \{\{x\}\ |\ x\in {\cal L}_2^2=L_{\iota }\setminus L_{\mu }\})\ |\ \theta \in {\rm Con}_{\V }(L_{\mu })\}=\{\Delta _{L_{\iota }},\nabla _{L_{\iota }}\}\cup \{eq(\{L_{\lambda }\}\cup \{\{x\}\ |\ x\in L_{\iota }\setminus L_{\lambda }\})\ |\ 2\leq \lambda \leq \mu \}$, hence the family $(L_{\lambda })_{2\leq \lambda \leq \iota }$ also satisfies condition \textcircled{c}$_{\V }$.

If $\iota $ is a limit ordinal such that the family $(L_{\lambda })_{2\leq \lambda <\iota }$ satisfies condition \textcircled{c}$_{\V }$, then, by (\ref{thelatconstr1}) and Lemma \ref{czedlis}, the family $(L_{\lambda })_{2\leq \lambda \leq \iota }$ also satisfies condition \textcircled{c}$_{\V }$.

By the induction principle, it follows that the family $(L_{\mu })_{2\leq \mu \leq \tau }$ satisfies condition \textcircled{c}$_{\V }$. By (\ref{thelatconstr1}) and Lemma \ref{elemncg}, (\ref{elemncg1}), it follows that $|{\rm Con}_{\V }(L_{\mu })|=|\mu |$ for all ordinals $\mu $ with $2\leq \mu \leq \tau $.

\noindent (\ref{thelatconstr5}) By (\ref{thelatconstr3}) and (\ref{thelatconstr4}), along with the description of the congruences in conditions \textcircled{c} and \textcircled{c}$_{\I }$ and the obvious fact that, if the 
lattice reduct of the i--lattice $L$ is simple, then so is $L$.

\noindent (\ref{thelatconstr2}) By Lemma \ref{elemncg}, (\ref{elemncg2}), we have $|L_{\mu }|=\nu $ for all $\mu \in I$. The property of the numbers of filters and ideals is trivial for $L_2=L$.

Now let $\iota \in I\setminus \{2\}$. Let us consider the ideal $J=(L]_{L_{\iota }}=(1^{L_3}]_{L_{\iota }}\setminus \{1^{L_3}\}=(0^{L_3}]_{L_{\iota }}\cup L=\{0^{L_{\lambda +1}}\ |\ 2\leq \lambda <\tau \}\cup L$ since the chain $(0^{L_3}]_{L_{\iota }}$ is formed of the elements $0^{L_{\mu }}$ with $\mu $ a successor ordinal in $I$. $J$ is a principal ideal of $L_{\iota }$ iff $L$ has a top element. Since the set $\{\mu \ |\ 2\leq \mu \leq \iota \}$ is well ordered and thus so is the filter $[1^{L_3})_{L_{\iota }}=\{1^{L_{\lambda +1}}\ |\ 2\leq \lambda <\iota \}$, it is easy to notice that $L_{\iota }$ satisfies the property \textcircled{g}$_J$, so, by Lemma \ref{alsoczedlis}, every nonprincipal filter of $L_{\iota }$ is generated by a filter of $J$.

Let $F$ be a nonprincipal filter of $L_{\iota }$. Then there exists a filter $G$ of $J=(0^{L_3}]_{L_{\iota }}\cup L$ such that $F=[G)_{L_{\iota }}$.

If $G\subseteq L$, so that $G$ is a filter of $L$, then $F=[G)_{L_{\iota }}=G\cup [1^{L_3})_{L_{\iota }}$, hence $G$ is nonprincipal since $F$ is nonprincipal.

If $G\nsubseteq L$, then $G\cap (J\setminus L)=G\cap (0^{L_3}]_{L_{\iota }}$ is nonempty, hence $H=G\cap (0^{L_3}]_{L_{\iota }}$ is a filter of $(0^{L_3}]_{L_{\iota }}$ and clearly, if we denote by $a_{L_{\mu }},b_{L_{\mu }}$ the two incomparable elements of the copy of ${\cal L}_2^2$ from $L_{\mu }=\B (L_{\lambda })\boxplus {\cal L}_2^2$ for each successor ordinal $\mu =\lambda +1$ with $\lambda \in I\setminus \{\tau \}$, then $\displaystyle F=[G)_{L_{\iota }}=[H)_{L_{\iota }}=H\cup \bigcup _{\mu \in I,0^{L_{\mu }}\in H}[0^{L_{\mu }},1^{L_{\mu }}]_{L_{\iota }}\cup L\cup [1^{L_3})_{L_{\iota }}=H\cup \{0^{L_{\mu }},a_{L_{\mu }},b_{L_{\mu }}\ |\ \mu \in I,0^{L_{\mu }}\in H\}\cup L\cup [1^{L_3})_{L_{\iota }}$, and thus $H$ is nonprincipal since $F$ is nonprincipal. $(0^{L_3}]_{L_{\iota }}=\{\lambda +1\ |\ 2\leq \lambda <\iota \}$ is dually well ordered, hence it has all ideals principal and thus, since it is a chain, $|{\rm Filt}((0^{L_3}]_{L_{\iota }})|=|{\rm Id}((0^{L_3}]_{L_{\iota }})|=|(0^{L_3}]_{L_{\iota }}|\leq |\iota |\leq |\tau |=\kappa \leq \nu $.

By the above, clearly, $G$ (thus also $G\cap (0^{L_3}]_{L_{\iota }}$ in the second case above) is uniquely determined by $F$, and hence $|{\rm Filt}(L_{\iota })|=|{\rm PFilt}(L_{\iota })|+|{\rm Filt}(L_{\iota })\setminus {\rm PFilt}(L_{\iota })|=|L_{\iota }|+|{\rm Filt}(J)\setminus {\rm PFilt}(J)|=\nu +|{\rm Filt}(J)\setminus {\rm PFilt}(J)|=|L|+|{\rm Filt}(L)\setminus {\rm PFilt}(L)|+|{\rm Filt}((0^{L_3}]_{L_{\iota }})\setminus {\rm PFilt}((0^{L_3}]_{L_{\iota }})|=|{\rm PFilt}(L)|+|{\rm Filt}(L)\setminus {\rm PFilt}(L)|+|{\rm Filt}((0^{L_3}]_{L_{\iota }})\setminus {\rm PFilt}((0^{L_3}]_{L_{\iota }})|=|{\rm Filt}(L)|+|{\rm Filt}((0^{L_3}]_{L_{\iota }})\setminus {\rm PFilt}((0^{L_3}]_{L_{\iota }})|=|{\rm Filt}(L)|$ since $|{\rm Filt}(L)|\geq |{\rm PFilt}(L)|=|L|=\nu \geq |{\rm Filt}((0^{L_3}]_{L_{\iota }})|\geq |{\rm Filt}((0^{L_3}]_{L_{\iota }})\setminus {\rm PFilt}((0^{L_3}]_{L_{\iota }})|$.

By duality, it follows that $|{\rm Id}(L_{\iota })|=|{\rm Id}(L)|$.\end{proof}

\begin{corollary} For any infinite simple (involution) lattice $L$ and every cardinal number $\kappa $ with $3\leq \kappa \leq |L|$, there exists a bounded (involution) lattice $M$ with $|M|=|L|$, $|{\rm Filt}(M)|=|{\rm Filt}(L)|$, $|{\rm Id}(M)|=|{\rm Id}(L)|$ and:\begin{enumerate}
\item\label{startfromL1} if $L$ is a simple lattice, then $|{\rm Con}(M)|=\kappa $;
\item\label{startfromL2} if $L$ is a simple i--lattice, then $|{\rm Con}_{\I }(M)|=\kappa $;
\item\label{startfromL3} if $L$ is an i--lattice with a simple lattice reduct, then ${\rm Con}_{\I }(M)={\rm Con}(M)$ and $|{\rm Con}_{\I }(M)|=|{\rm Con}(M)|=\kappa $;
\item\label{startfromL0} $L$ is an i--lattice and satisfies condition \textcircled{k}, then $M$ is a pseudo--Kleene algebra.\end{enumerate}\label{startfromL}\end{corollary}

\begin{proof} We apply to $L$ the construction above, with $L_2=L$ and $|\sigma |=|\tau |=\kappa $, take $M=L_{\tau }$ and apply Proposition \ref{thelatconstr} to obtain (\ref{startfromL1}), (\ref{startfromL2}) and (\ref{startfromL3}), then Remark \ref{pkas} to obtain (\ref{startfromL0}).\end{proof}

Note that the pseudo--Kleene algebras $L_{\nu ,\kappa }$ from Remark \ref{fromgccmth} have ${\rm Con}(L_{\nu ,\kappa })\cong {\rm Con}(M_{\nu ,\kappa })\times {\rm Con}(M_{\nu ,\kappa })^d={\rm Con}(M_{\nu ,\kappa })\times {\rm Con}(M_{\nu ,\kappa })$, hence $|{\rm Con}(L_{\nu ,\kappa })|=|{\rm Con}(M_{\nu ,\kappa })|^2=\kappa ^2$. Let us also obtain such pseudo--Kleene algebras with $\kappa $ many (involution--preserving) congruences, and, moreover, with their congruences coinciding to those of their lattice reducts: 

\begin{theorem} For any infinite cardinal number $\nu $, any cardinal number $\kappa $ with $2\leq \kappa \leq \nu $ or $\kappa =2^{\nu }$ and each $\mu \in \{\nu ,2^{\nu }\}$, there exists a bounded (involution) lattice $L_{\nu ,\mu ,\kappa }$ with $|{\rm Filt}(L_{\nu ,\mu ,\kappa })|=|{\rm Id}(L_{\nu ,\mu ,\kappa })|=\mu $, $|{\rm Con}(L_{\nu ,\mu ,\kappa })|=\kappa $ and, in the case when $L_{\nu ,\mu ,\kappa }$ is a bi--lattice, $|{\rm Con}_{\I }(L_{\nu ,\mu ,\kappa })|=|{\rm Con}(L_{\nu ,\mu ,\kappa })|=\kappa $ and, furthermore, $L_{\nu ,\mu ,\kappa }$ can be chosen to be a pseudo--Kleene algebra and, for $\kappa \leq \nu $, such that ${\rm Con}_{\I }(L_{\nu ,\mu ,\kappa })={\rm Con}(L_{\nu ,\mu ,\kappa })$.\label{mainth}\end{theorem}

\begin{proof} Let $T$ be a set with $|T|=\nu $ and let us consider the orthomodular lattice and thus pseudo--Kleene algebra $L_{\nu ,\nu ,2}={\cal M}_{\nu }={\cal M}_{\nu +\nu }=\boxplus _{t\in T}{\cal L}_2^2$, which has length $3$ and a simple lattice reduct, thus all filters and ideals principal and $|{\rm Con}_{\I }(L_{\nu ,\nu ,2})|=|{\rm Con}(L_{\nu ,\nu ,2})|=2$. By Corollary \ref{startfromL}, for every cardinal number $3\leq \kappa \leq \nu $, there exists a pseudo--Kleene algebra $L_{\nu ,\nu ,\kappa }$ with ${\rm Con}_{\I }(L_{\nu ,\nu ,\kappa })={\rm Con}(L_{\nu ,\nu ,\kappa })$, $|{\rm Con}_{\I }(L_{\nu ,\nu ,\kappa })|=|{\rm Con}(L_{\nu ,\nu ,\kappa })|=\kappa $ and $|{\rm Filt}(L_{\nu ,\nu ,\kappa })|=|{\rm Id}(L_{\nu ,\nu ,\kappa })|=|{\rm Filt}(L_{\nu ,\nu ,2})|=|{\rm Id}(L_{\nu ,\nu ,2})|=\nu $.

Now let $C$ be a well--ordered set with top element having $|C|=\nu $, thus a bounded chain with all filters principal and thus having $|{\rm Id}(C)|=|{\rm Filt}(C)|=|C|=\nu $. Since $C$ is a chain, each of its equivalences with all classes convex is a lattice congruence of $C$ \cite{gratzer,eucard,eunoucard}, hence, for any nonempty subset $S$ of $C\setminus \{0^C\}$, for each $a\in S\cup \{0^C\}$ such that $a$ is not a maximum for $S$, $a^+$ is the successor of $a$ in the well--ordered set $S\cup \{0^C\}$, and we denote by $\theta $ the equivalence on $C$ whose classes are the nonempty sets $[a,a^+]_C\setminus \{a^+\}$ for all $a$ as above, along with $\{x\in C\ |\ (\forall a\in S)\, (a<x)\}$ if this set is nonempty, then $\theta \in {\rm Con}(C)$. Therefore $2^{\nu }=2^{\nu -1}-1\leq |{\rm Con}(C)|\leq 2^{\nu }$, hence $|{\rm Con}(C)|=2^{\nu }$. Hence the Kleene chain $L_{\nu ,\nu ,2^{\nu }}=C\oplus C^d$ has $|{\rm Con}_{\I }(L_{\nu ,\nu ,2^{\nu }})|=|{\rm Con}(C)|=2^{\nu }$ and $|{\rm Filt}(L_{\nu ,\nu ,2^{\nu }})|=|{\rm Id}(L_{\nu ,\nu ,2^{\nu }})|=|{\rm Filt}(C)|+|{\rm Id}(C)|-1=\nu +\nu -1=\nu $.

Now we consider the pseudo--Kleene algebras $L_{\nu ,2^{\nu },2^{\nu }}=M_{\nu ,2^{\nu }}\oplus M_{\nu ,2^{\nu }}^d$ and $L_{\nu ,2^{\nu },2}=(M_{\nu ,2^{\nu }}\oplus M_{\nu ,2^{\nu }}^d)\boxplus {\cal L}_2^2$, where $M_{\nu ,2^{\nu }}$ is the following bounded lattice constructed in \cite[Example $5.6$]{eunoucard}: for a set $T$ with $|T|=\nu $, $M_{\nu ,2^{\nu }}$ is the bounded sublattice of the Boolean algebra ${\cal P}(T)\cong {\cal L}_2^{\nu }$ defined by $M_{\nu ,2^{\nu }}=\{(x_t)_{t\in T}\subseteq {\cal L}_2\ |\ |\{t\in T\ |\ x_t=1\}|<\aleph _0\}\cup \{1^{{\cal P}(T)}\}$. $|M_{\nu ,2^{\nu }}|=\nu $, hence $|L_{\nu ,2^{\nu },2^{\nu }}|=|L_{\nu ,2^{\nu },2}|=\nu $. $|{\rm Filt}(M_{\nu ,2^{\nu }})|=\nu $ and $|{\rm Id}(M_{\nu ,2^{\nu }})|=2^{\nu }$, thus $|{\rm Filt}(L_{\nu ,2^{\nu },2^{\nu }})|=|{\rm Id}(L_{\nu ,2^{\nu },2^{\nu }})|=|{\rm Filt}(L_{\nu ,2^{\nu },2})|=|{\rm Id}(L_{\nu ,2^{\nu },2})|=2^{\nu }$. $M_{\nu ,2^{\nu }}$ is distributive, hence it has more congruences than ideals and thus $|{\rm Con}_{\I }(L_{\nu ,2^{\nu },2^{\nu }})|=|{\rm Con}(M_{\nu ,2^{\nu }})|=2^{\nu }=2^{\nu }\cdot 2^{\nu }=|{\rm Con}(L_{\nu ,2^{\nu },2^{\nu }})|$. $M_{\nu ,2^{\nu }}$ is $0$--regular, so ${\rm Con}_0(M_{\nu ,2^{\nu }})=\{\Delta _{M_{\nu ,2^{\nu }}}\}$, thus ${\rm Con}_{01}(M_{\nu ,2^{\nu }}\oplus M_{\nu ,2^{\nu }}^d)=\{\Delta _{M_{\nu ,2^{\nu }}\oplus M_{\nu ,2^{\nu }}^d}\}$, hence ${\rm Con}_{\I }(L_{\nu ,2^{\nu },2})={\rm Con}(L_{\nu ,2^{\nu },2})=\{\Delta _{L_{\nu ,2^{\nu },2}},\nabla _{L_{\nu ,2^{\nu },2}}\}$.

By Corollary \ref{startfromL}, it follows that, for every cardinal number $3\leq \kappa \leq \nu $, there exists a pseudo--Kleene algebra $L_{\nu ,2^{\nu },\kappa }$ with ${\rm Con}_{\I }(L_{\nu ,2^{\nu },\kappa })={\rm Con}(L_{\nu ,2^{\nu },\kappa })$, $|{\rm Con}_{\I }(L_{\nu ,2^{\nu },\kappa })|=|{\rm Con}(L_{\nu ,2^{\nu },\kappa })|=\kappa $ and $|{\rm Filt}(L_{\nu ,2^{\nu },\kappa })|=|{\rm Id}(L_{\nu ,2^{\nu },\kappa })|=|{\rm Filt}(L_{\nu ,2^{\nu },2})|=|{\rm Id}(L_{\nu ,2^{\nu },2})|=2^{\nu }$.\end{proof}

\begin{corollary} For any infinite cardinal number $\nu $, any cardinal number $\kappa $ with $2\leq \kappa \leq \nu $ or $\kappa =2^{\nu }$ and each $\mu \in \{\nu ,2^{\nu }\}$, there exists an antiortholattice $A_{\nu ,\mu ,\kappa }$ with $|{\rm Filt}(A_{\nu ,\mu ,\kappa })|=|{\rm Id}(A_{\nu ,\mu ,\kappa })|=\mu $, $|{\rm Con}_{\BZ }(A_{\nu ,\mu ,\kappa })|=\kappa $ and such that, if $\kappa \leq \nu $, then ${\rm Con}_{\BZ }(A_{\nu ,\mu ,\kappa })={\rm Con}_{01}(A_{\nu ,\mu ,\kappa })\cup \{\nabla _{A_{\nu ,\mu ,\kappa }}\}$.\end{corollary}

\begin{proof} Let us consider the $0$--regular bounded lattice with $\nu $ elements and as many ideals as subsets $M_{\nu ,2^{\nu }}$ from the proof of Theorem \ref{mainth}. Then the antiortholattice $A_{\nu ,2^{\nu },2}=M_{\nu ,2^{\nu }}\oplus M_{\nu ,2^{\nu }}^d$ has $|A_{\nu ,2^{\nu },2}|=\nu $, $|{\rm Filt}(A_{\nu ,2^{\nu },2})|=|{\rm Id}(A_{\nu ,2^{\nu },2})|=2^{\nu }$ and $|{\rm Con}_{\BZ }(A_{\nu ,2^{\nu },2})|=2$.

Now we consider the simple $0$--regular bounded lattice ${\cal M}_{\nu }$ with $\nu $ elements and finite length, thus all filters and ideals principal, and we let $A_{\nu ,\nu ,2}={\cal M}_{\nu }\oplus {\cal M}_{\nu }^d$. Then the antiortholattice $A_{\nu ,\nu ,2}$ has $|A_{\nu ,\nu ,2}|=\nu $, $|{\rm Filt}(A_{\nu ,\nu ,2})|=|{\rm Id}(A_{\nu ,\nu ,2})|=\nu $ and $|{\rm Con}_{\BZ }(A_{\nu ,\nu ,2})|=2$.

Now let $\kappa $ be a cardinal number with $3\leq \kappa \leq \nu $ or $\kappa =2^{\nu }$ let $\mu \in \{\nu ,2^{\nu }\}$, and consider the antiortholattice $A_{\nu ,\mu ,\kappa }={\cal L}_2\oplus L_{\nu ,\mu ,\kappa -1}\oplus {\cal L}_2$, where $L_{\nu ,\mu ,\kappa -1}$ is a pseudo--Kleene algebra as in Theorem \ref{mainth}, that is with $\nu $ elements, $\mu $ filters and ideals, $\kappa -1$ congruences and, since $\kappa \leq \nu $ and thus $\kappa -1\leq \nu $, with all lattice congruences preserving its involution. Then $|A_{\nu ,\mu ,\kappa }|=\nu $, $|{\rm Filt}(A_{\nu ,\mu ,\kappa })|=|{\rm Id}(A_{\nu ,\mu ,\kappa })|=\mu $, $|{\rm Con}_{\BZ }(A_{\nu ,\mu ,\kappa })|=\kappa -1+1=\kappa $ and, if $\kappa \leq \nu $, then ${\rm Con}_{\BZ }(A_{\nu ,\mu ,\kappa })={\rm Con}_{\I 0}(A_{\nu ,\mu ,\kappa })\cup \{\nabla _{A_{\nu ,\mu ,\kappa }}\}=\{eq(L_{\nu ,\mu ,\kappa -1}/\beta \cup \{\{0\},\{1\}\})\ |\ \beta \in {\rm Con}_{\I }(L_{\nu ,\mu ,\kappa -1})\}\cup \{\nabla _{A_{\nu ,\mu ,\kappa }}\}=\{eq(L_{\nu ,\mu ,\kappa -1}/\beta \cup \{\{0\},\{1\}\})\ |\ \beta \in {\rm Con}(L_{\nu ,\mu ,\kappa -1})\}\cup \{\nabla _{A_{\nu ,\mu ,\kappa }}\}={\rm Con}_{01}(A_{\nu ,\mu ,\kappa })\cup \{\nabla _{A_{\nu ,\mu ,\kappa }}\}$.\end{proof}

\begin{corollary} Under the Generalized Continuum Hypothesis:\begin{itemize}
\item an infinite (bounded) lattice with any numbers of filters and ideals can have any number of congruences between $2$ and its number of subsets;
\item a pseudo--Kleene algebra with any number of ideals can have any number of congruences between $2$ and its number of subsets and, simultaneously, when it has strictly less congruences than subsets, its congruences coinciding to those of its lattice reduct;
\item an antiortholattice with any number of ideals can have any number of congruences between $2$ and its number of subsets and, simultaneously, when it has strictly less congruences than subsets, its proper congruences coinciding to the congruences of its lattice reduct that have singleton classes of its lattice bounds.\end{itemize}

Under the Continuum Hypothesis, the above hold for countable bounded lattices, pseudo--Kleene algebras, respectively antiortholattices.\end{corollary}

\section*{Acknowledgements}

I thank G\' abor Cz\' edli for insightful discussions on some of the topics of this paper.

This work was supported by the research grant {\em Propriet\`a d`Ordine Nella Semantica Algebrica delle Logiche Non--classiche} of Universit\`a degli Studi di Cagliari, Regione Autonoma della Sardegna, L. R. $7/2007$, n. $7$, $2015$, CUP: ${\rm F}72{\rm F}16002920002$.

\end{document}